\documentclass[11pt]{article}

\usepackage{ amsthm, amsmath,amssymb}

\def \[{\begin{equation}}
\def \]{\end{equation}}

\def\bdes{\begin{description}}
\def\edes{\end{description}}
\def\benu{\begin{enumerate}}
\def\eenu{\end{enumerate}}
\def\bitm{\begin{itemize}}
\def\eitm{\end{itemize}}

\def\R{{\sl I\kern-3.2pt R}}

\def\sqr#1#2{{\vcenter{\hrule height .#2pt
      \hbox{\vrule width .#2pt height#1pt \kern#1pt\vrule width.#2pt}
                       \hrule height.#2pt}}}

\newtheorem{lemma}{Lemma}[section]

\newtheorem{remark}{Remark}[section]

\newtheorem{thm}{Theorem}[section]

\begin{document}

\title{{ \Large \bf Symmetry and nonexistence of positive solutions for fractional
systems}
}
\author{\;\;Pei Ma, Yan Li, Jihui Zhang\thanks{Corresponding author.}}

\date{\today}
\maketitle

\begin{abstract}
This paper is devoted to study the nonexistence of positive
solutions for the following fractional H$\acute{e}$non
system
\begin{eqnarray*}\left\{
\begin{array}{lll}
 &(-\triangle)^{\alpha/2}u=|x|^av^p,~~~&x\in R^n,\\
 &(-\triangle)^{\alpha/2}v=|x|^bu^q,~~~ &x\in R^n,\\
 &u\geq0, v\geq 0,
 \end{array}
 \right.
\end{eqnarray*}
where $0<\alpha<2$, $1\leq p,q<\infty$,
$a$, $b$ $\geq0$, $n\geq2$.
Using a direct method of moving planes, we prove the non-existence of positive solution in the subcritical case.
\end{abstract}


\section{Introduction}
We study the following system involving the fractional Laplacian
\begin{eqnarray}\label{1.1}
\left\{
\begin{array}{ll}
 &(-\triangle)^{\alpha/2}u=|x|^av^p,\,\,x\in R^n\\
 &(-\triangle)^{\alpha/2}v=|x|^bu^q,\,\,x\in R^n\\
 &u\geq0, v\geq 0,\,\,x\in R^n,
 \end{array}
 \right.
\end{eqnarray}
 where $0<\alpha<2$, $1\leq p,q<\infty$, $a, b\geq0$, $n\geq2$.

 The fractional Laplacian in $R^n$ is a nonlocal pseudo-differential operator taking the form
\begin{eqnarray}\label{1.2}
(-\triangle)^{\alpha/2}u(x)=C_{n,\alpha}PV
\int_{R^n}\frac{u(x)-u(y)}{|x-y|^{n+\alpha}}dy=C_{n,\alpha}\lim_{\varepsilon\rightarrow
0} \int_{R^n\setminus
B_\varepsilon(x)}\frac{u(x)-u(y)}{|x-y|^{n+\alpha}}dy,
\end{eqnarray}
where $C_{n,\alpha}$ is a normalization constant. This operator is well defined in
$\mathcal{S}$, the Schwartz space of rapidly decreasing $C^\infty$
functions in $R^n$. In this space, it can also be equivalently
defined in terms of the Fourier transform
\begin{eqnarray*}
\mathcal{F}[(-\triangle)^{\alpha/
2}u](\xi)=|\xi|^\alpha\mathcal{F}u(\xi),
\end{eqnarray*}
where $\mathcal{F}u$ is the Fourier transform of $u$. One can extend
this operator to a wider space of distributions:
\begin{eqnarray*}
\mathcal{L}_\alpha= \{u:R^n\rightarrow
R\mid\int_{R^n}\frac{|u(x)|}{1+|x|^{n+\alpha}}dx<\infty\}.
\end{eqnarray*}
Then in this space, we defined $(-\triangle)^{\alpha/2}u$ as a
distribution by

\begin{eqnarray*}
\langle(-\triangle)^{\alpha/2}u(x), \phi\rangle=\int_{R^n}
u(x)(-\triangle)^{\alpha/2}\phi(x)dx, \,\,\forall\phi\in C_0^\infty(R^n).
\end{eqnarray*}

In our paper, we use the direct method of moving planes
introduced by \cite{CLL} to derive the symmetry of positive solutions
and then deduce the nonexistence of positive solutions under explain how this paper improved previous result by listing their conditions \cite{DZ}. The following is our main theorems.

\begin{thm} \label{t1}
Let $(u,v)\in L_\alpha\cap L_{loc}^\infty$ be a pair of nonnegative
solution to (\ref{1.1}) and $1<p<\frac{n+\alpha+a}{n-\alpha}$,
$1<q<\frac{n+\alpha+b}{n-\alpha}$. Then $u$ and $v$ are radially
symmetric and decreasing about the origin.
\end{thm}

\begin{thm}\label{t2}
Assume $a+\alpha>0$, $b+\alpha>0$ and $1\leq p, q<\infty$, let
$(u,v)\in L_\alpha\cap L_{loc}^\infty$ be a pair of nonnegative
solution to (\ref{1.1}), then $u$ and $v$ also satisfy
\begin{eqnarray}\label{i1}
\left\{
\begin{array}{ll}
 &u(x)=C_0\int_{R^n}\frac{|y|^av^p(y)}{|x-y|^{n-\alpha}}dy,\\
 &v(x)=C_1\int_{R^n}\frac{|y|^bu^q(y)}{|x-y|^{n-\alpha}}dy,
 \end{array}
 \right.
\end{eqnarray}
and vice versa, where $C_0$ and $C_1$ are positive constants.
\end{thm}
Using the above theorems, we can show the following nonexistence of
positive solutions to (\ref{1.1}) or (\ref{i1}).
\begin{thm}\label{t3}
Let $(u,v)\in L_\alpha\cap L_{loc}^\infty$ be a pair of non-negative
solution to (\ref{1.1}) or (\ref{i1}),
$1<p<\frac{n+\alpha+a}{n-\alpha}$ and
$1<q<\frac{n+\alpha+b}{n-\alpha}$ . Then $(u,v)=(0,0)$.
\end{thm}

As a consequence of the proof in Theorem \ref{t3}, we obtain the
following nonexistence result immediately.

\begin{thm}\label{t4}
Let $(u,v)\in L_\alpha\cap L_{loc}^\infty$ be a pair of non-negative
solution to (\ref{1.1}) or (\ref{i1}). If
\begin{equation*}
\int_{R^n}\frac{|y|^av^p(y)}{|x-y|^{n-\alpha}}dy<\infty\,\, and\,\,
\int_{R^n}\frac{|y|^bu^q(y)}{|x-y|^{n-\alpha}}dy<\infty
\end{equation*}
with
$\frac{n+b}{q+1}+\frac{n+a}{p+1}\neq n-\alpha$. Then $(u,v)=(0,0)$.
\end{thm}

\begin{remark}
In \cite{DZ}, the authors applied the method of moving planes in
integral forms to study the symmetry of positive solutions for the
fractional system (\ref{1.1}), then derived the
nonexistence of positive solutions. However, due to technical restrictions to apply the integral method, the authors have to assume that $u\in L^\beta_{loc}(R^n)$ and $v\in L^\gamma_{loc}(R^n)$
where $\beta=\frac{n(q-1)}{(n-\alpha)q-b-n}$ and $\gamma=\frac{n(p-1)}{(n-\alpha)p-a-n}$. In this paper,we manage to derive the same nonexistence result without imposing extra integrability conditions, by using a direct application of the moving of the  moving planes to the differential equations.
\end{remark}

In recent years, the fractional H$\acute{e}$non-type problem has received
a lot of attention.
In \cite{LNZ1}, the authors considered the integral equation
\begin{eqnarray*}
 u(x)=\int_{R^n}\frac{u^p(y)}{|x-y|^{n-\alpha}|y|^s}dy,
\end{eqnarray*}
where $\frac{n-s}{n-\alpha}<p<\alpha^*(s)-1$ with $\alpha^*(s)=\frac{2(n-s)}{n-\alpha}$.
They proved the nonexistence of positive solutions for the equation by the method of moving planes in integral forms and established the equivalence between the above integral equation and the following partial differential equation
$$(-\triangle)^{\frac{\alpha}{2}}u(x)=|x|^{-s}u^p.$$

In \cite{LNZ}, the authors studied the following weighted system of partial differential equations
\begin{eqnarray*}\left\{
\begin{array}{lll}
 &(-\triangle)^{\alpha/2}u=|x|^{-s}v^p,~~~&in ~~~R^n,\\
 &(-\triangle)^{\alpha/2}v=|x|^{-t}u^q,~~~ &in ~~~R^n,\\
 &u\geq0, v\geq 0.
 \end{array}
 \right.
\end{eqnarray*}
They first established the equivalence between the differential system
and an integral system
\begin{eqnarray*}
\left\{
\begin{array}{ll}
 &u(x)=\int_{R^n}\frac{v^p(y)}{|x-y|^{n-\alpha}|y|^s}dy,\\
 &v(x)=\int_{R^n}\frac{u^q(y)}{|x-y|^{n-\alpha}|y|^t}dy.
 \end{array}
 \right.
\end{eqnarray*}
Then, in the critical case $\frac{n-s}{p+1}+\frac{n-t}{q+1}=n-\alpha$, they showed that every pair of positive
solutions $(u(x), v(x))$ is radially symmetric about the origin. While in the subcritical
case, they proved the nonexistence of positive solutions.

In \cite{ZCCY}, the authors considered the following fractional Laplacian equation
\begin{eqnarray*}
\left\{
\begin{array}{ll}
 (-\triangle)^{\alpha/2}u(x)=0,\,\,in\,\,R^n\\
 u(x)\geq 0,\,\,in\,\,R^n,
 \end{array}
 \right.
\end{eqnarray*}
where $n\geq 2$ and $\alpha$ is any real number between 0 and 2. They proved
that the only solution is constant. As an application, they obtained an equivalence
between a semi-linear differential equation
$$(-\triangle)^{\alpha/2}u(x)=u^p(x),\,\,x\in R^n,$$
and the corresponding integral equation
$$u(x)=\int_{R^n}\frac{u^p(y)}{|x-y|^{n-\alpha}}dy,\,\,x\in R^n.$$
For more similar results, please see \cite{CL2}, \cite{CDL}, \cite{F} and the references therein.

The paper is organized as follows. In Section 2, we use the method of moving
planes to prove Theorem 1.1. In Section
3, we establish the equivalence between problem $(\ref{1.1})$ and
integral system (\ref{i1}). Finally, in Section 4 we  complete the proof of
Theorem \ref{t3} and \ref{t4}.
\section{The Symmetry of Positive Solutions}

Without any decay conditions on $u$ and $v$ ,
we are not able to carry the method of moving planes on $u$ and $v$
directly. To circumvent this difficulty, we make a Kelvin transform.
Let
\begin{eqnarray*}
\left\{
\begin{array}{ll}
&\overline{u}(x)=\frac{1}{|x|^{n-\alpha}}u(\frac{x}{|x|^2}),\\
&\overline{v}(x)=\frac{1}{|x|^{n-\alpha}}v(\frac{x}{|x|^2}).\\
\end{array}
 \right.
\end{eqnarray*}
Then
\begin{eqnarray*}
(-\triangle)^{\alpha/2}\overline{u}(x)&=&\frac{1}{|x|^{n+\alpha}}(-\triangle)^{\alpha/2}u(\frac{x}{|x|^2}),\\
&=&\frac{1}{|x|^{n+\alpha+a}}v^p(\frac{x}{|x|^2}),\\
&=&\frac{1}{|x|^{n+\alpha+a-p(n-\alpha)}}\overline{v}^p(x).
\end{eqnarray*}
  In a similar way, we have
\begin{eqnarray*}
(-\triangle)^{\alpha/2}\overline{v}(x)=\frac{1}{|x|^{n+\alpha+b-q(n-\alpha)}}\overline{u}^q(x).
\end{eqnarray*}
Let $\gamma=n+\alpha+a-p(n-\alpha)$,
$\beta=n+\alpha+b-q(n-\alpha)$, and $(\ref{1.1})$ becomes
\begin{eqnarray}\label{1.3}
\left\{
\begin{array}{ll}
 &(-\triangle)^{\alpha/2}\overline{u}=|x|^{-\gamma}\overline{v}^p,\\
 &(-\triangle)^{\alpha/2}\overline{v}=|x|^{-\beta}\overline{u}^q,\\
 &\overline{u}\geq0, \overline{v}\geq 0.
 \end{array}
 \right.
\end{eqnarray}

We first give some basic notations before starting moving the planes. Then we start moving planes on system $(\ref{1.3})$.

Let
\begin{eqnarray*}
T_\lambda=\{x\in R^n\mid x_1=\lambda,\lambda\in R\}
\end{eqnarray*}
be the moving plane,
\begin{eqnarray*}
\Sigma_\lambda=\{x\in R^n\mid x_1<\lambda\}
\end{eqnarray*}
be the region to the left of the plane, and
\begin{eqnarray*}
x^\lambda=(2\lambda-x_1,x_2,...,x_n)
\end{eqnarray*}
be the reflection of the point $x = (x_1, x_2, \cdot\cdot\cdot ,
x_n)$ about the plane $\mathrm{T}_\lambda$.

Assume that $(\overline{u},\overline{v})$ solves the
fractional system $(\ref{1.3})$. To compare the values of
$\overline{u}(x)$ with $\overline{u}(x^\lambda)$ and
$\overline{v}(x)$ with $\overline{v}(x^\lambda)$, we denote
\begin{eqnarray*}
\left\{
\begin{array}{ll}
U_\lambda(x) = \overline{u}(x^\lambda)-\overline{u}(x),\\
V_\lambda(x) = \overline{v}(x^\lambda)-\overline{v}(x).\\
\end{array}
 \right.
\end{eqnarray*}
Then system $(\ref{1.3})$ becomes

\begin{eqnarray}\label{1.4}
\left\{
\begin{array}{ll}
(-\triangle)^{\alpha/2} U_{\lambda}(x)=\frac{1}{|x^\lambda|^{\gamma}}\overline{v}^{p}(x^\lambda)-\frac{1}{|x|^{\gamma}}\overline{v}^{p}(x),\\
(-\triangle)^{\alpha/2}
V_{\lambda}(x)=\frac{1}{|x^\lambda|^{\beta}}\overline{u}^{q}(x^\lambda)-\frac{1}{|x|^{\beta}}\overline{u}^{q}(x).
\end{array}
 \right.
\end{eqnarray}

Before starting moving planes, we need some lemmas in \cite{CLL}.
\begin{lemma}\label{L2.1}$\mathbf{(Narrow~~region~~ principle  \cite{CLL})}$ Let $\Omega$ be a bounded narrow region in $\Sigma_\lambda$, such that it is contained in $\{x|\lambda-l<x_1<\lambda\}$ with small $l$. Suppose that $\varphi \in L_\alpha\bigcap C_{loc}^{1,1}(\Omega)$ and is lower semi-continuous on $\overline{\Omega}$. If $C(x)$ is bounded from below in $\Omega$, then
\begin{eqnarray}\label{1.14}
\left\{
\begin{array}{lll}
(-\triangle)^{\alpha/2}\varphi(x)+C(x)\varphi(x)\geq 0 ~~~ ~~~&in& \Omega,\\
 \varphi(x)\geq0 ~~~ ~~~&in& \Sigma_\lambda\backslash\Omega,\\
\varphi(x^\lambda)=-\varphi(x)   ~~~ ~~~&in& \Sigma_\lambda,
\end{array}
 \right.
\end{eqnarray}
then for sufficiently small $\delta$, we have
\begin{eqnarray}
\varphi(x) \geq 0 ~~~\text{in} ~~~\Omega.
\end{eqnarray}
Furthermore, if $\varphi = 0$ at some point in $\Omega$, then
\begin{eqnarray}
\varphi(x) = 0 ~~~\text{almost everywhere in}~~~ R_n.
\end{eqnarray}
These conclusions hold for unbounded region $\Omega$ if we further
assume that
\begin{eqnarray}
\underline{\lim\limits}_{|x|\rightarrow\infty}\varphi(x)\geq0.
\end{eqnarray}
\end{lemma}

\begin{lemma}\label{L2.2}$\mathbf{(Decay~~at~~infinity \cite{CLL})}$ Let $\Omega$ be an unbounded region in $\Sigma_\lambda$. Assume $\varphi\in L_\alpha\cap C^{1,1}_{loc}(\Omega)$ is a solution of
\begin{eqnarray}\label{1.20}
\left\{
\begin{array}{lll}
(-\triangle)^{\alpha/2}\varphi(x)+C(x)\varphi(x)\geq 0 ~~~ ~~~&in& \Omega,\\
 \varphi(x)\geq0 ~~~ ~~~&in& \Sigma_\lambda\backslash\Omega,\\
\varphi(x^\lambda)=-\varphi(x)   ~~~ ~~~&in& \Sigma_\lambda,
\end{array}
 \right.
\end{eqnarray}
with
\begin{eqnarray}
\underline{\lim\limits}_{|x|\rightarrow\infty}|x|^\alpha C(x)\geq0,
\end{eqnarray}
then there exists a constant $R_0 > 0$ such that if
\begin{eqnarray}
\varphi(x_0) = \min_{\Omega}\varphi(x)<0,
\end{eqnarray}
then
\begin{eqnarray}
|x_0|\leq R_0.
\end{eqnarray}
\end{lemma}

\subsection{Proof of Theorem 1.1}
Now,we start moving planes.

$\mathbf{Step.1}$: we show that when $\lambda$ sufficiently
negative,
\begin{eqnarray}\label{1.8}
 U_\lambda(x), V_\lambda(x)\geq0, \forall x\in\Sigma_{\lambda}\setminus\{0^\lambda\}.
\end{eqnarray}

Let $W_\lambda(x)=U_\lambda(x)+V_\lambda(x)$, it follows from$(\ref{1.4})$

\begin{eqnarray*}
(-\triangle)^{\alpha/2}W_\lambda(x)&=&(-\triangle)^{\alpha/2}(U_\lambda(x)-V_\lambda(x)),\\
&=&\frac{1}{|x^\lambda|^{\gamma}}\overline{v}^{p}(x^\lambda)-\frac{1}{|x|^{\gamma}}\overline{v}^{p}(x)+
\frac{1}{|x^\lambda|^{\beta}}\overline{u}^{q}(x^\lambda)-\frac{1}{|x|^{\beta}}\overline{u}^{q}(x).
\end{eqnarray*}

We claim that $W_\lambda(x)\geq 0$. If not there exists $x_0\in\Sigma_{\lambda}\setminus\{0^\lambda\}$, such that
$W_\lambda(x_0)=\text{min}_{\sum_\lambda\setminus\{0^\lambda\}}W_\lambda(x)<0$. This is guaranteed by the fact that, for $\epsilon$ sufficiently small and $\lambda$ sufficiently negative, it holds that
$$U_\lambda(x), V_\lambda(x)\geq C>0, x\in B_\varepsilon(0^\lambda)\backslash\{0^\lambda\},$$
we will prove it in the Appendices.
Without loss of generality, we suppose that $U_\lambda(x_0)<0$, then we have
$V_\lambda(x_0)<0$.

To see this, we suppose that $V_\lambda(x_0)\geq0$. It follows from $U_\lambda(x_0)<0$ that there exists $x_1$, such that
$U_\lambda(x_1)=\text{min}_{\sum_\lambda\setminus\{0^\lambda\}}U_\lambda(x)<0$.
Thus

\begin{eqnarray}\label{1.5}\nonumber
(-\triangle)^{\alpha/2}U_{\lambda}(x_1)&=&C_{n,\alpha}PV
\int_{R^n}\frac{U_{\lambda}(x_1)-U_{\lambda}(y)}{|x_1-y|^{n+\alpha}}dy,\\ \nonumber
&=&C_{n,\alpha}PV
\int_{\sum_\lambda}\frac{U_{\lambda}(x_1)-U_{\lambda}(y)}{|x_1-y|^{n+\alpha}}dy+C_{n,\alpha}PV
\int_{R^n\backslash\sum_\lambda}\frac{U_{\lambda}(x_1)-U_{\lambda}(y)}{|x_1-y|^{n+\alpha}}dy,\\
\nonumber &=&C_{n,\alpha}PV
\int_{\sum_\lambda}\frac{U_{\lambda}(x_1)-U_{\lambda}(y)}{|x_1-y|^{n+\alpha}}dy+C_{n,\alpha}PV
\int_{\sum_\lambda}\frac{U_{\lambda}(x_1)-U_{\lambda}(y^\lambda)}{|x_1-y^\lambda|^{n+\alpha}}dy,\\
\nonumber &=&C_{n,\alpha}PV
\int_{\sum_\lambda}\frac{U_{\lambda}(x_1)-U_{\lambda}(y)}{|x_1-y|^{n+\alpha}}dy+C_{n,\alpha}PV
\int_{\sum_\lambda}\frac{U_{\lambda}(x_1)+U_{\lambda}(y)}{|x_1-y^\lambda|^{n+\alpha}}dy,\\
\nonumber &\leq&
C_{n,\alpha}\int_{\sum_\lambda}\frac{U_{\lambda}(x_1)-U_{\lambda}(y)}{|x_1-y^\lambda|^{n+\alpha}}dy+
\frac{U_{\lambda}(x_1)+U_{\lambda}(y)}{|x_1-y^\lambda|^{n+\alpha}}dy,\\ \nonumber
&=&C_{n,\alpha}\int_{\sum_\lambda}\frac{2U_{\lambda}(x_1)}{|x_1-y^\lambda|^{n+\alpha}}dy.\\
&<&0.
\end{eqnarray}
On the other hand,
\begin{eqnarray}\label{1.6} \nonumber
(-\triangle)^{\alpha/2}U_{\lambda}(x_1)
&=&\frac{1}{|x_1^\lambda|^{\gamma}}\overline{v}^{p}(x_1^\lambda)-\frac{1}{|x_1|^{\gamma}}\overline{v}^{p}(x_1),\\
\nonumber
&=&\frac{1}{|x_1^\lambda|^{\gamma}}\overline{v}^{p}(x_1^\lambda)-\frac{1}{|x_1|^\gamma}\overline{v}^p(x_1^\lambda)+
\frac{1}{|x_1|^\gamma}\overline{v}^p(x_1^\lambda)-\frac{1}{|x_1|^{\gamma}}\overline{v}^{p}(x_1),\\
\nonumber
&=&\frac{1}{|x_1|^{\gamma}}(\overline{v}^{p}(x_1^\lambda)-\overline{v}^p(x_1))+
\overline{v}^{p}(x_1^\lambda)(\frac{1}{|x_1^\lambda|^\gamma}-\frac{1}{|x_1|^{\gamma}}),\\
&\geq&\frac{1}{|x_1|^{\gamma}}(\overline{v}^{p}(x_1^\lambda)-\overline{v}^p(x_1)).
\end{eqnarray}
By $(\ref{1.5})$ and $(\ref{1.6})$, we can deduce that
\begin{eqnarray*}
V_\lambda(x_1)<0.
\end{eqnarray*}
Then
\begin{eqnarray*}
W_\lambda(x_1)<
W_\lambda(x_0),
\end{eqnarray*}
which is a contradiction. This proves that $V_\lambda(x_0)<0$.\\
Notice that

\begin{eqnarray*}
(-\triangle)^{\alpha/2}W_\lambda(x_0)&\geq&\frac{1}{|x_0|^\gamma}(\overline{v}^p(x_0^\lambda)-
\overline{v}^p(x_0))+\frac{1}{|x_0|^\beta}(\overline{u}^q(x_0^\lambda)-\overline{u}^q(x_0)),\\
&=&\frac{1}{|x_0|^\gamma}p\xi^{p-1}(\overline{v}(x_0^\lambda)-\overline{v}(x_0))
+\frac{1}{|x_0|^\beta}q\eta^{q-1}(\overline{u}(x_0^\lambda)-\overline{u}(x_0))\\
&\geq&\frac{1}{|x_0|^\gamma}p\overline{v}^{p-1}(x_0)(\overline{v}(x_0^\lambda)-\overline{v}(x_0))
+\frac{1}{|x_0|^\beta}q\overline{u}^{q-1}(x_0)(\overline{u}(x_0^\lambda)-\overline{u}(x_0))\\
&\geq&\frac{1}{|x_0|^\gamma}p\overline{v}^{p-1}(x_0)\mathrm{V}_\lambda(x_0)
+\frac{1}{|x_0|^\beta}q\overline{u}^{q-1}(x_0)\mathrm{U}_\lambda(x_0)
\end{eqnarray*}

where $\xi\in(\overline{v}(x_0^\lambda),\overline{v}(x_0))$,
$\eta\in(\overline{u}(x_0^\lambda),\overline{u}(x_0))$. Let
$C_1=\frac{1}{|x_0|^\gamma}p\overline{v}^{p-1}(x_0)$,
$\mathrm{C}_2=\frac{1}{|x_0|^\beta}q\overline{u}^{q-1}(x_0)$,
$C=\text{max}\{C_1,C_2\}$, then

\begin{eqnarray}\label{1.7}
(-\triangle)^{\alpha/2}W_\lambda(x_0)-CW_\lambda(x_0)\geq0.
\end{eqnarray}
Similar to$(\ref{1.5})$, we have

\begin{eqnarray}\label{8888}
(-\triangle)^{\alpha/2}W_\lambda(x_0)-CW_\lambda(x_0) \leq
W_\lambda(x_0)(\int_{\sum_\lambda}
\frac{1}{|x_0-y^\lambda|^{n+\alpha}}dy-C(x_0))<0.
\end{eqnarray}
Indeed,
\begin{eqnarray}\label{1.34}\nonumber
\int_{\sum_\lambda}
\frac{1}{|x_0-y^\lambda|^{n+\alpha}}dy&\geq&\int_{\{x_1\geq0\}}
\frac{1}{|x_0-y^\lambda|^{n+\alpha}}dy\\ \nonumber
&=&\frac{1}{2}\int_{R^n}\frac{1}{(|x_0|+|y^\lambda|)^{n+\alpha}}dy\\
\nonumber
&=&\frac{1}{2}\int_{0}^{\infty}\int_{B_r^0}\frac{1}{(|x_0|+|r|)^{n+\alpha}}d\sigma
dr\\ \nonumber
&=&\frac{1}{2}\int_0^{\infty}\frac{w_{n-1}r^{n-1}}{(|x_0|+|r|)^{n+\alpha}}dr\\
\nonumber
&=&\frac{w_{n-1}}{2|x_0|^\alpha}\int_0^{\infty}\frac{t^{n-1}}{(1+t)^{n+\alpha}}dr,           r=t|x_2|\\
&\sim&\frac{C_3}{|x_0|^\alpha},
\end{eqnarray}
and
\begin{eqnarray*}
C_1&=&\frac{1}{|x_0|^\gamma}p\overline{v}^{p-1}(x_0)\\
&\sim&p\frac{1}{|x_0|^\gamma}(\frac{1}{|x_0|^{n-\alpha}})^{p-1}\\
&\sim&\frac{1}{|x_0|^{\gamma+(p-1)(n-\alpha)}},
\end{eqnarray*}
\begin{eqnarray*}
C_2&=&\frac{1}{|x_0|^\beta}q\overline{u}^{q-1}(x_0)\\
&\sim&q\frac{1}{|x_0|^\beta}(\frac{1}{|x_0|^{n-\alpha}})^{q-1}\\
&\sim&\frac{1}{|x_0|^{\beta+(q-1)(n-\alpha)}}.
\end{eqnarray*}
Note that $\gamma+(p-1)(n-\alpha)=2\alpha+a>\alpha$ and $\beta+(q-1)(n-\alpha)=2\alpha+b>\alpha$, then
$$\int_{\sum_\lambda}\frac{1}{|x_0-y^\lambda|^{n+\alpha}}dy-C(x_0)>0.$$

(\ref{8888}) is a contraction with $(\ref{1.7})$. This shows that $W_\lambda(x)\geq0$. In a similar process as in (\ref{1.5}) and (\ref{1.6}) we can prove
$$U_\lambda(x)\geq0,\,\,V_\lambda(x)\geq0.$$

$\mathbf{Step.2}$: Step 1 provides a starting point, from which we
can now move
the plane $\mathrm{T_\lambda}$ to the right as long as $(\ref{1.8})$ holds to its limiting position.\\
Let
\begin{eqnarray*}
\lambda_0=\text{sup}\{\lambda\leq 0|U_\mu\geq0, V_\mu\geq0,
\forall x\in \Sigma_\mu\setminus\{0^\lambda\} ,\mu\leq\lambda \}
\end{eqnarray*}
In this part, we show that
\begin{eqnarray*}
\lambda_0=0
\end{eqnarray*}
Suppose that
\begin{eqnarray*}
\lambda_0<0
\end{eqnarray*}
we show that the plane $\mathrm{T}_\lambda$ can be moved further
right. To be more rigorous, there exists some $\delta>0$, such
that for any $\lambda\in(\lambda_0,\lambda_0 + \delta)$, we have
\begin{eqnarray}\label{1.10}
W_{\lambda}(x)\geq0, ~~ x\in
\Sigma_{\lambda}\setminus\{0^\lambda\}.
\end{eqnarray}
This is a contradiction with the definition of $\lambda_0$.
In fact, when $\lambda_0<0$ , we have
\begin{eqnarray}\label{1.36}
W_{\lambda_0}(x)>0, ~~ x\in \Sigma_{\lambda_0}\setminus
\{0^\lambda\}.
\end{eqnarray}
If not, there exists some $\hat{x}$ such that
\begin{eqnarray*}
W_{\lambda_0}(\hat{x}) = \min_{x\in \Sigma_{\lambda_0}\setminus
\{0^\lambda\}} W_{\lambda_0}(x) = 0, i.e.
U_{\lambda_0}(\hat{x}) =0, V_{\lambda_0}(\hat{x}) =0.
\end{eqnarray*}
It follows that
\begin{eqnarray}\label{1.37}\nonumber
(-\triangle)^{\alpha/2}U_{\lambda_0}(\hat{x})&=&C_{n,\alpha}PV
\int_{\mathbb{R}^n}\frac{-U_{\lambda_0}(y)}{|\hat{x}-y|^{n+\alpha}}dy\\
\nonumber &=&C_{n,\alpha}PV
\int_{\Sigma_{\lambda_0}}\frac{-U_{\lambda_0}(y)}{|\hat{x}-y|^{n+\alpha}}dy+C_{n,\alpha}PV
\int_{\widetilde{\Sigma_{\lambda_0}}}\frac{-U_{\lambda_0}(y)}{|\hat{x}-y|^{n+\alpha}}dy\\
\nonumber &=&C_{n,\alpha}PV
\int_{\Sigma_{\lambda_0}}\frac{-U_{\lambda_0}(y)}{|\hat{x}-y|^{n+\alpha}}dy+C_{n,\alpha}PV
\int_{\Sigma_{\lambda_0}}\frac{-U_{\lambda_0}(y^\lambda)}{|\hat{x}-y^\lambda|^{n+\alpha}}dy\\
\nonumber &=&C_{n,\alpha}PV
\int_{\Sigma_{\lambda_0}}\frac{-U_{\lambda_0}(y)}{|\hat{x}-y|^{n+\alpha}}dy+C_{n,\alpha}PV
\int_{\Sigma_{\lambda_0}}\frac{U_{\lambda_0}(y)}{|\hat{x}-y^\lambda|^{n+\alpha}}dy\\
\nonumber
&=& C_{n,\alpha}PV\int_{\Sigma_{\lambda_0}}(\frac{1}{|\hat{x}-y^\lambda|^{n+\alpha}}-\frac{1}{|\hat{x}-y|^{n+\alpha}}) U_{\lambda_0}(y) dy\\
&\leq&0.
\end{eqnarray}
On the other hand
\begin{equation}\label{1.38}
(-\triangle)^{\alpha/2}U_{\lambda_0}(\hat{x})=\frac{1}{|\hat{x}^\lambda|^{\gamma}}\overline{v}^{p}(\hat{x}^\lambda)-\frac{1}{|\hat{x}|^{\gamma}}\overline{v}^{p}(\hat{x})=
\frac{1}{|\hat{x}^\lambda|^{\gamma}}\overline{v}^{p}(\hat{x})-\frac{1}{|\hat{x}|^{\gamma}}\overline{v}^{p}(\hat{x})>0
\end{equation}
A contradiction with $(\ref{1.37})$, so does $V_{\lambda_0}$. This
proves $(\ref{1.36})$. We  claim that for $\lambda_0<0$ and $\varepsilon>0$ sufficiently small,
$$U_{\lambda_0}(x), V_{\lambda_0}(x)\geq C>0, x\in B_\varepsilon(0^{\lambda_0})\setminus\{0^{\lambda_0}\},$$
the proof will be given in the appendix. It follows from $(\ref{1.36})$ that there exists
a constant $C_0>0$, such that
\begin{eqnarray*}
 W_{\lambda_0}(x)\geq C_0>0, ~~x \in \overline{\Sigma_{\lambda_0-\delta}\cap B_{R}(0)}\backslash\{0^{\lambda_0}\}.
\end{eqnarray*}
Then
\begin{eqnarray}\label{1.31}
W_\lambda(x) \geq 0,~~\forall x\in \Sigma_{\lambda_0-\delta}\cap B_{R}(0)
\setminus\{0^\lambda\}.
\end{eqnarray}

From decay at infinity, we have
\begin{eqnarray}\label{1.39}
W_\lambda(x)\geq0, ~~\forall x\in B_R^c.
\end{eqnarray}

From narrow region principle,
\begin{eqnarray}\label{1.32}
W_\lambda(x)\geq0, \forall x\in (\Sigma_\lambda\setminus
\Sigma_{\lambda_0- \delta})\setminus \{0^\lambda\}.
\end{eqnarray}
To see this, in Lemma 2.1, we let, for any sufficiently small $\eta
> 0$, $H = \Sigma_\lambda \setminus B_\eta(0^\lambda)$ and the
narrow region $\Omega = (\Sigma^-_\lambda \setminus
\Sigma_{\lambda_0-\delta})\setminus B_\eta(0^\lambda)$, while
the lower bound of $C(x)$ can be seen from $(\ref{1.34})$.

Combining $(\ref{1.31})$ $(\ref{1.32})$ and $(\ref{1.39})$, we conclude that
$$W_\lambda(x)\geq 0,~~x\in\Sigma_\lambda\backslash\{0^\lambda\}.$$

This contradicts the definition of $\lambda_0$. Therefore, we must
have $\lambda_0=0$ and $W_{\lambda_0}\geq0, \forall x \in
\Sigma_\lambda$. Similarly, one can move the plane $T_\lambda$ from
the $+\infty$ to the left and show that $W_{\lambda_0}\leq0, \forall
x \in \Sigma_\lambda$. Now we have shown that
\begin{eqnarray}
\lambda_0=0, W_{\lambda_0}\equiv0, \forall x \in \Sigma_\lambda.
\end{eqnarray}
\subsection{Proof of theorem \ref{t1}}
\begin{proof}

So far, we have proved that $\overline{u},\overline{v}$ is
symmetric about the plane $T_0$. Since the $x_1$ direction can be
chosen arbitrarily, we have actually shown that
$\overline{u},\overline{v}$ is radially symmetric about $0$. Let $x_1$ $x_2$ be any points centered at $0$, i.e.,
$$0=\frac{x_1+x_2}{2},|x_1|=|x_2|.$$
Then,
$$\overline{u}(x_1)=\overline{u}(x_2),\overline{v}(x_1)=\overline{v}(x_2).$$
Let
$$y_1=\frac{x_1}{|x_1|^2},y_2=\frac{x_2}{|x_2|^2},$$
then
$$\frac{y_1+y_2}{2}=0.$$
Hence,
$$u(y_1)=u(\frac{x_1}{|x_1|^2})=|x_1|^{n-\alpha}\overline{u}(x_1)=|x_2|^{n-\alpha}\overline{u}(x_2)=u(y_2),$$
$$v(y_1)=u(\frac{x_1}{|x_1|^2})=|x_2|^{n-\alpha}\overline{v}(x_1)=|x_2|^{n-\alpha}\overline{v}(x_2)=v(y_2).$$
This completes the proof.
\end{proof}

\section{The equivalence between problem (\ref{1.1}) and the integral form (\ref{i1})}

In this section, we prove the equivalence between problem
(\ref{1.1}) and (\ref{i1}).

\textbf{Proof of Theorem \ref{t2}} Let $(u,v)$ be a pair of positive
solution to (\ref{1.1}), we first show that
\begin{eqnarray}\label{2.1}
\left\{
\begin{array}{ll}
 &u(x)=c_1+\int_{R^n}\frac{|y|^av^p(y)}{|x-y|^{n-\alpha}}dy,\\
 &v(x)=c_2+\int_{R^n}\frac{|y|^bu^q(y)}{|x-y|^{n-\alpha}}dy.
 \end{array}
 \right.
\end{eqnarray}
Let
\begin{eqnarray}\label{2.2}
\left\{
\begin{array}{ll}
 &u_R(x)=\int_{B_R(0)}G_R(x,y)|y|^av^p(y)dy,\\
 &v_R(x)=\int_{B_R(0)}G_R(x,y)|y|^bu^q(y)dy,
 \end{array}
 \right.
\end{eqnarray}
where $G_R(x,y)$ is the Green function of fractional Laplacian on
$B_R(0)$.

It is easy to see that
\begin{eqnarray}\label{2.3}
\left\{
\begin{array}{ll}
 (-\triangle)^{\alpha/2}u_R(x)=|x|^av^p(x),\,\,in\,\,B_R(0),\\
 (-\triangle)^{\alpha/2}v_R(x)=|x|^bu^q(x),\,\,in\,\,B_R(0),\\
 u_R(x)=v_R(x)=0,\,\,\,\,\,on\,\,B_R^c(0).
 \end{array}
 \right.
\end{eqnarray}
Let $\varphi_R(x)=u(x)-u_R(x)$ and $\psi_R(x)=v(x)-v_R(x)$. From
(\ref{1.1}) and (\ref{2.3}), we have
\begin{eqnarray*}
\left\{
\begin{array}{ll}
 (-\triangle)^{\alpha/2}\varphi_R(x)=0,\,\,in\,\,B_R(0),\\
 (-\triangle)^{\alpha/2}\phi_R(x)=0,\,\,in\,\,B_R(0),\\
 \varphi_R(x), \phi_R(x)\geq 0,\,\,\,\,\,on\,\,B_R^c(0).
 \end{array}
 \right.
\end{eqnarray*}
By the Maximum Principle, we derive
\begin{eqnarray}\label{2.4}
   \varphi_R(x), \phi_R(x)\geq 0,\,\,\,x\in R^n.
\end{eqnarray}
Therefore, when $R\rightarrow \infty,$

\begin{eqnarray}\label{2.5}
\left\{
\begin{array}{ll}
u_R(x)\rightarrow \tilde{u}(x)=\int_{R^n}\frac{|y|^av^p(y)}{|x-y|^{n-\alpha}}dy,\\
v_R(x)\rightarrow
\tilde{v}(x)=\int_{R^n}\frac{|y|^bu^q(y)}{|x-y|^{n-\alpha}}dy,
\end{array}
\right.
\end{eqnarray}

Moreover,
\begin{eqnarray}\label{2.6}
\left\{
\begin{array}{ll}
(-\triangle)^{\alpha/2}\tilde{u}(x)=|y|^av^p(y),\,\,\,x\in R^n,\\
(-\triangle)^{\alpha/2}\tilde{v}(x)=|y|^bu^q(y),\,\,\,x\in R^n.
\end{array}
 \right.
\end{eqnarray}
Now let $\Phi(x)=u(x)-\tilde{u}(x)$ and $\Psi(x)=v(x)-\tilde{v}(x)$.
From (\ref{1.1}) and (\ref{2.6}), we have
\begin{eqnarray*}
\left\{
\begin{array}{ll}
(-\triangle)^{\alpha/2}\Phi(x)=0,\,\,\,x\in R^n,\\
(-\triangle)^{\alpha/2}\Psi(x)=0,\,\,\,x\in R^n,\\
\Phi(x), \Psi(x)\geq 0,\,\,\,\,x\in R^n.
\end{array}
 \right.
\end{eqnarray*}
From Proposition 2 in \cite{ZCCY}, we have
$$\Phi(x)=c_1, \Psi(x)=c_2.$$
Thus we proved (\ref{2.1}).

Next, we will show that $c_1=c_2=0$. Without lose of generality, we
may assume that $c_2>0$, then from (\ref{2.1}), we have
$$u(x)=c_1+\int_{R^n}\frac{|y|^av^p(y)}{|x-y|^{n-\alpha}}dy\geq c_1+\int_{R^n}\frac{c_2|y|^a}{|x-y|^{n-\alpha}}dy=\infty.$$
But it is impossible, hence $c_1=c_2=0$. Therefore
\begin{eqnarray*}
\left\{
\begin{array}{ll}
 &u(x)=\int_{R^n}\frac{|y|^av^p(y)}{|x-y|^{n-\alpha}}dy,\\
 &v(x)=\int_{R^n}\frac{|y|^bu^q(y)}{|x-y|^{n-\alpha}}dy.
 \end{array}
 \right.
\end{eqnarray*}
We complete our proof.

\section{The nonexistence of positive solutions}
In this section, we prove Theorem \ref{t3} and Theorem
\ref{t4}. First we need some Lemmas.
\begin{lemma}\label{lemma3.1}
Assume $(u,v)$ is a pair of positive radial solution for (\ref{1.1})
then for $r=|x|>0$ it holds
\begin{eqnarray}\label{3.3}
u(r)\leq Cr^{-\frac{(b+\alpha)p+a+\alpha}{pq-1}},\\ \label{3.4}
v(r)\leq Cr^{-\frac{(a+\alpha)q+b+\alpha}{pq-1}}.
\end{eqnarray}
\end{lemma}
\begin{proof}
From (\ref{i1}) and the decreasing property of radial solution, we
have
\begin{eqnarray}\nonumber
u(r)&=&\int_{R^n}\frac{|y|^av^p(y)}{|x-y|^{n-\alpha}}dy\\ \nonumber
&\geq&\int_{B_r(0)}\frac{|y|^av^p(y)}{|x-y|^{n-\alpha}}dy\\
\nonumber
&\geq&\int_{B_r(0)}\frac{|y|^av^p(r)}{|x-y|^{n-\alpha}}dy\\
\nonumber
&\geq&|x|^nv^p(r)\int_{B_1(0)}\frac{|z|^a|x|^a}{|x-z|x||^{n-\alpha}}dy\\
\nonumber &=&Cv^p(r)|x|^{a+\alpha}\\ \label{3.1}
&=&Cv^p(r)r^{a+\alpha}.
\end{eqnarray}

Similarly, we have
\begin{eqnarray}\label{3.2}
 v(r)\geq Cu^q(r)r^{b+\alpha}.
\end{eqnarray}

Combining (\ref{3.1}) and (\ref{3.2}), it gives
\begin{eqnarray*}
  u(r) &\geq& C[u^q(r)r^{b+\alpha}]^pr^{a+\alpha} \\
   &=& Cu^{pq}r^{(b+\alpha)p+a+\alpha},
\end{eqnarray*}
and
\begin{eqnarray*}
  v(r) &\geq& C[v^p(r)r^{a+\alpha}]^qr^{b+\alpha} \\
   &=& Cv^{pq}r^{(a+\alpha)q+b+\alpha}.
\end{eqnarray*}
(\ref{3.3}) and (\ref{3.4}) follow immediately from the
above inequalities.
\end{proof}
\begin{lemma}
Assume $(u,v)$ is a pair of positive radial solution for
(\ref{1.1}), then it holds
\begin{eqnarray}\label{3.5}
  \int_{R^n}|x|^bu^{q+1}dx< \infty, &and& \int_{R^n}|x|^bv^{q+1}(x\cdot\nabla v(x))dx<\infty, \\ \label{3.6}
  \int_{R^n}|x|^av^{p+1}dx< \infty, &and& \int_{R^n}|x|^au^{p+1}(x\cdot\nabla u(x))dx<\infty.
\end{eqnarray}
\end{lemma}
\begin{proof}
For any $R>0$, we need to show
\begin{eqnarray*}
  \int_{B_R(0)}|x|^bv^{q+1}dx< \infty,\,\,as\,\, R\rightarrow\infty,\\
  \int_{B_R(0)}|x|^au^{p+1}dx< \infty,\,\,as\,\, R\rightarrow\infty.
\end{eqnarray*}
Here we only show $\int_{B_R(0)}|x|^bv^{q+1}dx< \infty,
R\rightarrow\infty$, the other can be proved in a similar way. The integral will converge only when
$$|x|^bu^{q+1}\sim o(\frac{1}{|x|^n})\,\,for\,\, |x| \,\,large.$$
Since
$$|x|^bu^{q+1}\leq|x|^b(|x|^{-\frac{(b+\alpha)p+a+\alpha}{pq-1}})^{q+1},$$
it is sufficient to show that
\begin{eqnarray}\label{3.7}
\frac{(b+\alpha)p+a+\alpha}{pq-1}(q+1)-b>n.
\end{eqnarray}
Or
\begin{equation}\label{mape111}
p(n-\alpha)(q-\frac{b+\alpha}{n-\alpha})\leq(n+b)+(a+\alpha)(q+1).
\end{equation}
Case i: if $q\leq\frac{b+\alpha}{n-\alpha}$, (\ref{mape111}) is automatically
true.

Case ii: if $q\geq\frac{b+\alpha}{n-\alpha}$, then (\ref{mape111}) becomes
\begin{eqnarray}\label{3.8}
  \frac{p(n-\alpha)}{a+\alpha}<\frac{q+\frac{n+b}{a+\alpha}+1}{q-\frac{b+\alpha}{n-\alpha}}
  =1+\frac{\frac{n+b}{a+\alpha}+\frac{b+\alpha}{n-\alpha}+1}{q-\frac{b+\alpha}{n-\alpha}}.
\end{eqnarray}
Notice that $1<p<\frac{n+\alpha+a}{n-\alpha}$ and
$1<q<\frac{n+\alpha+b}{n-\alpha}$,
$$LHS~~~ of ~~~(\ref{3.8})\leq\frac{n+\alpha+a}{n-\alpha}\cdot\frac{n-\alpha}{a+\alpha}=1+\frac{n}{a+\alpha}.$$
And
\begin{eqnarray*}
  RHS~~~ of~~~ (\ref{3.8})&\geq&1+\frac{1+\frac{n+b}{a+\alpha}+\frac{b+\alpha}{n-\alpha}}{\frac{n+b+\alpha}{n-\alpha}-\frac{b+\alpha}{n-\alpha}} \\
  &=& 1+\frac{1+\frac{n+b}{a+\alpha}+\frac{b+\alpha}{n-\alpha}}{\frac{n}{n-\alpha}} \\
   &=& 1+\frac{n+b}{a+\alpha}\cdot\frac{n-\alpha}{n}+\frac{n-\alpha}{n}+\frac{b+\alpha}{n} \\
  &=& 1+\frac{n+b}{n}\cdot\frac{n+a}{a+\alpha}.
\end{eqnarray*}
(\ref{3.8}) is true as long as
\begin{eqnarray}\label{3.9}
\frac{n}{a+\alpha}<\frac{n+b}{n}\cdot\frac{n+a}{a+\alpha},
\end{eqnarray}
since $a, b >0$ and $a+\alpha>0$. This completes the proof.
\end{proof}
\textbf{Proof of Theorem \ref{t3}} Let $(u,v)\in L_\alpha\cap
L_{loc}^\infty$ be a pair of positive solution to (\ref{1.1}) or
(\ref{i1}). By (\ref{i1}), we have
\begin{eqnarray}\label{3.10}
\left\{
\begin{array}{ll}
 &u(k x)=\int_{R^n}\frac{|y|^av^p(y)}{|k x-y|^{n-\alpha}}dy,\\
 &v(k x)=\int_{R^n}\frac{|y|^bu^q(y)}{|k x-y|^{n-\alpha}}dy.
 \end{array}
 \right.
\end{eqnarray}

We differentiate the first equation of (\ref{3.10}) with respect to
k,
$$x\cdot\nabla u(kx)=(\alpha-n)\int_{R^n}\frac{x\cdot(kx-y)|y|^av^p(y)}{|kx-y|^{n-\alpha+2}}dy,\,\,x\neq0.$$
Let $k=1$, then
\begin{eqnarray}\label{3.11}
x\cdot\nabla
u(x)=(\alpha-n)\int_{R^n}\frac{x\cdot(x-y)|y|^av^p(y)}{|x-y|^{n-\alpha+2}}dy,\,\,x\neq0.
\end{eqnarray}
Multiply both sides of (\ref{3.11}) by $|x|^bu^q(x)$ and integrate
on $R^n$, we have
\begin{eqnarray*}
 \int_{R^n}|x|^bu^q(x)(x\cdot\nabla u(x))dx =(\alpha-n)\int_{R^n}\int_{R^n}\frac{x\cdot(x-y)|x|^bu^q(x)|y|^av^p(y)}{|x-y|^{n-\alpha+2}}dydx.
\end{eqnarray*}

On the other hand, from the integration by parts formula, it follows
\begin{eqnarray*}
  &&\int_{B_R(0)}|x|^bu^q(x)(x\cdot\nabla u(x))dx \\
  &=& \frac{1}{q+1} \int_{B_R(0)}|x|^b(x\cdot\nabla u^{q+1}(x))dx\\
  &=& -\frac{n+b}{q+1}\int_{B_R(0)}|x|^bu^{q+1}(x)dx+\frac{1}{p+1}\int_{\partial B_R(0)}R^{b+1}u^{q+1}d\sigma.
\end{eqnarray*}
From Lemma \ref{lemma3.1}, we have
$$\frac{1}{p+1}\int_{\partial B_R(0)}R^{b+1}u^{q+1}d\sigma\rightarrow 0,\,\,R\rightarrow \infty.$$
Hence when $R\rightarrow \infty$,
\begin{eqnarray*}
 \int_{R^n}|x|^bu^q(x)(x\cdot\nabla u(x))dx=-\frac{n+b}{q+1}\int_{R^n}|x|^bu^{q+1}(x)dx .
\end{eqnarray*}
Therefore,
\begin{eqnarray}\label{3.12}
-\frac{n+b}{q+1}\int_{R^n}|x|^bu^{q+1}(x)dx
=(\alpha-n)\int_{R^n}\int_{R^n}\frac{x\cdot(x-y)|x|^bu^q(x)|y|^av^p(y)}{|x-y|^{n-\alpha+2}}dydx.
\end{eqnarray}
Using the similar argument on the second equation of (\ref{3.10}),
we also have
\begin{eqnarray*}
  \int_{R^n}|x|^av^p(x)(x\cdot\nabla v(x))dx =(\alpha-n)\int_{R^n}\int_{R^n}\frac{x\cdot(x-y)|x|^av^p(x)|y|^bu^q(y)}{|x-y|^{n-\alpha+2}}dydx,
\end{eqnarray*}
and
\begin{eqnarray*}
 \int_{R^n}|x|^av^p(x)(x\cdot\nabla v(x))dx=-\frac{n+a}{p+1}\int_{R^n}|x|^av^{p+1}(x)dx .
\end{eqnarray*}
Consequently,
\begin{eqnarray}\label{3.13}
-\frac{n+a}{p+1}\int_{R^n}|x|^av^{p+1}(x)dx
=(\alpha-n)\int_{R^n}\int_{R^n}\frac{x\cdot(x-y)|x|^av^p(x)|y|^bu^q(y)}{|x-y|^{n-\alpha+2}}dydx.
\end{eqnarray}
Adding (\ref{3.12}) and (\ref{3.13}) together, and using the fact
$|x-y|^2=x\cdot(x-y)+y\cdot(y-x)$, we have
\begin{eqnarray}\nonumber
&&-\frac{n+a}{p+1}\int_{R^n}|x|^av^{p+1}(x)dx
-\frac{n+b}{q+1}\int_{R^n}|x|^bu^{q+1}(x)dx\\ \nonumber &=&
(\alpha-n)\int_{R^n}\int_{R^n}\frac{x\cdot(x-y)|x|^bu^q(x)|y|^av^p(y)}{|x-y|^{n-\alpha+2}}dydx\\
\nonumber
&&+(\alpha-n)\int_{R^n}\int_{R^n}\frac{x\cdot(x-y)|x|^av^p(x)|y|^bu^q(y)}{|x-y|^{n-\alpha+2}}dydx\\
\nonumber &=&
\frac{\alpha-n}{2}\int_{R^n}\int_{R^n}\frac{x\cdot(x-y)|x|^bu^q(x)|y|^av^p(y)}{|x-y|^{n-\alpha+2}}dydx\\
\nonumber
&&+\frac{\alpha-n}{2}\int_{R^n}\int_{R^n}\frac{y\cdot(y-x)|y|^bu^q(y)|x|^av^p(x)}{|x-y|^{n-\alpha+2}}dxdy\\
\nonumber
&&+\frac{\alpha-n}{2}\int_{R^n}\int_{R^n}\frac{x\cdot(x-y)|x|^av^p(x)|y|^bu^q(y)}{|x-y|^{n-\alpha+2}}dydx\\
\nonumber
    &&+\frac{\alpha-n}{2}\int_{R^n}\int_{R^n}\frac{y\cdot(y-x)|y|^av^p(y)|x|^bu^q(x)}{|x-y|^{n-\alpha+2}}dxdy\\ \nonumber
&=&
\frac{\alpha-n}{2}\int_{R^n}\int_{R^n}\frac{x\cdot(x-y)|x|^bu^q(x)|y|^av^p(y)}{|x-y|^{n-\alpha+2}}dydx\\
\nonumber
  &&+\frac{\alpha-n}{2}\int_{R^n}\int_{R^n}\frac{y\cdot(y-x)|y|^bu^q(y)|x|^av^p(x)}{|x-y|^{n-\alpha+2}}dydx\\ \nonumber
&&+\frac{\alpha-n}{2}\int_{R^n}\int_{R^n}\frac{x\cdot(x-y)|x|^av^p(x)|y|^bu^q(y)}{|x-y|^{n-\alpha+2}}dydx\\
\nonumber
    &&+\frac{\alpha-n}{2}\int_{R^n}\int_{R^n}\frac{y\cdot(y-x)|y|^av^p(y)|x|^bu^q(x)}{|x-y|^{n-\alpha+2}}dydx\\ \nonumber
&=&
\frac{\alpha-n}{2}\int_{R^n}\int_{R^n}\frac{|x|^av^p(x)|y|^bu^q(y)}{|x-y|^{n-\alpha}}dydx\\
\nonumber
    &&+\frac{\alpha-n}{2}\int_{R^n}\int_{R^n}\frac{|y|^av^p(y)|x|^bu^q(x)}{|x-y|^{n-\alpha}}dydx \\ \label{3.14}
&=& \frac{\alpha-n}{2}\int_{R^n}|x|^av^{p+1}(x)dx+
\frac{\alpha-n}{2}\int_{R^n}|x|^bu^{q+1}(x)dx,
\end{eqnarray}
since
\begin{eqnarray*}
 \int_{R^n}|x|^av^{p+1}(x)dx &=& \int_{R^n}\frac{|x|^av^p(x)|y|^bu^q(y)}{|x-y|^{n-\alpha}}dydx \\
 &=& \int_{R^n}\frac{|y|^av^p(y)|x|^bu^q(x)}{|x-y|^{n-\alpha}}dxdy \\
   &=& \int_{R^n}\frac{|y|^av^p(y)|x|^bu^q(x)}{|x-y|^{n-\alpha}}dydx \\
  &=& \int_{R^n}|x|^bu^{q+1}(x)dx.
\end{eqnarray*}
Then (\ref{3.14}) becomes
$$(\frac{\alpha-n}{2}+\frac{n+b}{q+1})\int_{R^n}|x|^bu^{q+1}(x)dx
+(\frac{\alpha-n}{2}+\frac{n+a}{p+1})\int_{R^n}|x|^av^{p+1}(x)dx=0.$$
That is
$$(\alpha-n+\frac{n+b}{q+1}+\frac{n+a}{p+1})\int_{R^n}|x|^bu^{q+1}(x)dx=
(\alpha-n+\frac{n+b}{q+1}+\frac{n+a}{p+1})\int_{R^n}|x|^av^{p+1}(x)dx
=0.$$ Because $1<p<\frac{n+\alpha+a}{n-\alpha}$ and $1<q<\frac{n+\alpha+b}{n-\alpha}$, thus $ \frac{n+b}{q+1}+\frac{n+a}{p+1}\neq n-\alpha$ and
problem (\ref{1.1}) admits no positive solutions.

This completes our proof.

\textbf{Proof of Theorem \ref{t4}} Theorem 1.4 is a a direct consequence of Theorem 1.3.
\section{Appendices}
\begin{lemma}\label{lem5.1}
For $\lambda$ negative large, there exists a constant $C>0$, such that
\begin{equation}\label{5.1}
U_\lambda(x), V_\lambda(x)\geq C>0, x\in B_\varepsilon(0^\lambda)\backslash\{0^\lambda\}.
\end{equation}
\end{lemma}
\begin{proof}
For $x\in\Sigma_\lambda$, as $\lambda\rightarrow -\infty$, it is easy to see that
\begin{equation}\label{5.2}
\overline{u}(x)\rightarrow 0.
\end{equation}
To prove (\ref{5.1}), it is sufficient to show
$$\overline{u}_\lambda(x)\geq C>0, x\in B_\varepsilon(0^\lambda)\backslash\{0^\lambda\}.$$
Or equivalently,
$$\overline{u}(x)\geq C>0, x\in B_\varepsilon(0)\backslash\{0\}.$$
Let $\eta$ be a smooth cut-off function such that $\eta\in[0,1]$ in $R^n$, $supp~ \eta\subset B_2$ and $\eta\equiv 1$ in $B_1$. Let
$$(-\triangle)^{\alpha/2}\phi(x)=\eta(x)|x|^av^p(x).$$
Then,
$$\phi(x)=C_{n,-\alpha}\int_{R^n}\frac{\eta(y)|y|^av^p(y)}{|x-y|^{n-\alpha}}dy
=C_{n,-\alpha}\int_{B_2(0)}\frac{\eta(y)|y|^av^p(y)}{|x-y|^{n-\alpha}}dy.$$
It is trivial for $|x|$ sufficiently large,
\begin{equation}\label{5.3}
  \phi(x)\sim\frac{1}{|x|^{n-\alpha}}.
\end{equation}
Since
\begin{eqnarray}\label{5.4}\left\{
\begin{array}{lll}
 &(-\triangle)^{\alpha/2}(u-\phi)\geq 0,~~~&x\in B_R,\\
 &(u-\phi)(x)\geq 0,~~~ &x\in B_R^c,
 \end{array}
 \right.
\end{eqnarray}
by the maximum principle, we have
$$(u-\phi)(x)\geq 0,x\in B_R,$$
thus
$$(u-\phi)(x)\geq0,x\in R^n.$$
For $|x|$ sufficiently large, from (\ref{5.3}), one can see that for some constant $C>0$,
\begin{equation}\label{5.5}
 u(x)\geq\frac{C}{|x|^{n-\alpha}}.
\end{equation}
Hence for $|x|$ small
$$u(\frac{x}{|x|^2})\geq C|x|^{n-\alpha},$$
and
$$\overline{u}(x)=\frac{1}{|x|^{n-\alpha}}u(\frac{x}{|x|^2})\geq C.$$
Together with (\ref{5.2}), it yields that
\begin{equation}\label{5.6}
  U_\lambda(x)\geq\frac{C}{2}>0, x\in B_\varepsilon(0^\lambda)\setminus\{0^\lambda\}.
\end{equation}
Through an identical argument, one can show that (5.6) holds for $V_\lambda(x)$ as well.
\end{proof}
\begin{lemma}\label{lem5.2}
For $\lambda_0<0$, if either of $U_{\lambda_0}$ $V_{\lambda_0}$ is not identically 0, then there exist some constant $C$ and $\varepsilon>0$ small such that
$$U_{\lambda_0}(x), V_{\lambda_0}(x)\geq C>0, x\in B_\varepsilon(0^{\lambda_0})\setminus\{0^{\lambda_0}\}.$$
\end{lemma}
\begin{proof} From Lemma 2.2 in \cite{CFY}, we have the integral equation
\begin{eqnarray*}
  U_{\lambda_0}(x)&=& \overline{u}_{\lambda_0}(x)-\overline{u}(x) \\
   &=& C_{n,\alpha}\int_{\Sigma_{\lambda_0}}(|y^{\lambda_0}|^{-\gamma}\overline{v}^p(y^{\lambda_0})-|y|^{-\gamma}\overline{v}^p(y))
   (\frac{1}{|x-y|^{n+\alpha}}-\frac{1}{|x-y^{\lambda_0}|^{n+\alpha}})dy \\
   &\geq& C_{n,\alpha}\int_{\Sigma_{\lambda_0}}\frac{\overline{v}_{\lambda_0}^p(y)-\overline{v}^p(y)}{|y|^\gamma}
   \cdot(\frac{1}{|x-y|^{n+\alpha}}-\frac{1}{|x-y^{\lambda_0}|^{n+\alpha}})dy
\end{eqnarray*}
Since
$$V_{\lambda_0}(x)\not\equiv 0,x\in\Sigma_{\lambda_0},$$
there exists some $x_0$ such that
$V_{\lambda_0}(x_0)>0.$
Thus, for some $\delta>0$ small, it holds that
$$\overline{v}_{\lambda_0}^p(y)-\overline{v}^p(y)\geq C>0, y\in B_\delta(x_0).$$
Therefore,
\begin{equation}\label{5.7}
 U_{\lambda_0}(x)\geq\int_{B_\delta(x_0)}Cdy\geq C>0.
\end{equation}
In a same way, one can show that $V_{\lambda_0}(x)$ also satisfies (\ref{5.7}).
\end{proof}
\noindent{\bf Acknowledgement}

The research was supported by NSFC(NO.11571176) and Natural Science
Foundation of the Jiangsu Higher Education Institutions
(No.14KJB110017). The authors would like to express sincere thanks
to the anonymous referee for his/her carefully reading the
manuscript and valuable comments and suggestions.

\bigskip

{\em Author's Addresses and Emails:}
\medskip

Pei Ma

Jiangsu Key Laboratory for NSLSCS

School of Mathematical Sciences

Nanjing Normal University

Nanjing, Jiangsu 210023, China;

Department of Mathematical Sciences

Yeshiva University

New York, NY, 10033, USA

mapei0620@126.com

\medskip

Yan Li

Yeshiva University

New York, NY, 10033, USA

yali3@mail.yu.edu

\medskip

Jihui Zhang

Jiangsu Key Laboratory for NSLSCS

School of Mathematical Sciences

Nanjing Normal University

Nanjing, Jiangsu 210023, China

zhangjihui@njnu.edu.cn
\end{document}